\documentclass[12pt,a4paper]{amsart}
\usepackage{amsfonts,a4wide,amsmath,amssymb,amsthm}
\usepackage[utf8]{inputenc}  
\usepackage[T1]{fontenc}
\usepackage[all]{xy}    
\usepackage{mathtools}  
\usepackage{multirow}    
\usepackage{csquotes}	
\usepackage{bm} 
\usepackage{tikz}   
\usepackage{float}
\usepackage{relsize}
\usepackage{tikz-cd}
\usepackage{bbold}
\usepackage[pagebackref=true]{hyperref}      
\usepackage{comment}
\numberwithin{equation}{section}
\allowdisplaybreaks

\usepackage{anyfontsize}

\theoremstyle{plain}
\newtheorem{thm}{Theorem}[section]
\newtheorem{prop}{Proposition}[section]
\newtheorem{cor}[prop]{Corollary}
\newtheorem{question}{Question}[section]

\newtheorem{lem}[prop]{Lemma}

\theoremstyle{definition}
\newtheorem{defi}[prop]{Definition}

\theoremstyle{remark}
 \newtheorem{rem}[prop]{Remark}

\newcommand{\Ccal}{{\mathcal C}}
\newcommand{\Dcal}{{\mathcal D}}

\newcommand{\Fcal}{{\mathcal F}}

\newcommand{\Hcal}{{\mathcal H}}

\newcommand{\Ocal}{{\mathcal O}}

\newcommand{\N}{{\mathbb{N}}}
\newcommand{\Z}{{\mathbb{Z}}}
\newcommand{\Q}{{\mathbb{Q}}}
\newcommand{\R}{{\mathbb{R}}}
\newcommand{\C}{{\mathbb{C}}}

\newcommand{\GL}{\operatorname{GL}}
\newcommand{\SL}{\operatorname{SL}}

\renewcommand{\Im}{\operatorname{Im}}

\newcommand{\Cl}{\operatorname{Cl}}

\renewcommand{\Re}{\operatorname{Re}}

\newcommand{\Jac}{\operatorname{Jac}}

\newcommand{\Sp}{\operatorname{Sp}}

\newcommand{\Tr}{\operatorname{Tr}}

\newcommand{\vol}{\operatorname{vol}}

\newcommand{\dd}{\textrm{d}}

\newcommand{\ra}{\rightarrow}

\author{Linda Frey}
\address{Linda Frey: Mathematisches Institut der Universit\"at G\"ottingen, Bunsenstrasse 3-5, 37073 G\"ottingen, Germany}
\email{lindafrey89@gmail.com}

\author{Samuel Le Fourn}
 \address{Samuel Le Fourn: Institut Fourier, UMR 5582, Université Grenoble Alpes, France}
 \email{samuel.le-fourn@univ-grenoble-alpes.fr}

\author{Elisa Lorenzo Garc\'ia}
 \address{Elisa Lorenzo Garc\'ia: Institut de Math\'ematiques, Universit\'e de Neuch\^atel, Rue Emile-Argand 11, 2000, Neuch\^atel, Switzerland -- Laboratoire IRMAR, Office 602,
Universit\'e de Rennes 1, Campus de Beaulieu, 35042, Rennes Cedex} \email{elisa.lorenzo@unine.ch, elisa.lorenzogarcia@univ-rennes1.fr }

\title{Explicit height estimates for CM curves of genus 2 }
\begin{document}
\begin{abstract} In this paper we make explicit the constants of Habegger and Pazuki's work from 2017 on bounding the discriminant of cyclic Galois CM fields corresponding to genus 2 curves with CM by them and potentially good reduction outside a predefined set of primes. We also simplify some of the arguments.
\end{abstract}


\maketitle

\section{Introduction}

Serre and Tate's work in 1968  \cite{ST68} demonstrated that an abelian variety with complex multiplication exhibits potentially good reduction everywhere. This principle applies not only to elliptic curves but also to Jacobians of curves of genus $g\geq2$. However, it is important to note that these curves can still encounter geometrically bad reduction.

In 1983, Faltings provided a groundbreaking result when he proved the Shafarevich Conjecture \cite{Falt83}. This conjecture implies, when we fix the field of definition for our varieties, the existence of at most finitely many curves of a given genus with potentially good reduction outside a predefined set of primes. This can be viewed as the analog result for curves in a manner similar to the classical result in number fields, which states that there are only a finite number of unramified extensions of bounded degree outside a fixed set of primes. It is important to emphasize that this is while keeping the base field fixed.

Habegger and Pazuki, in their recent work \cite{HabeggerPazuki17} from 2017, provided an elegant finitness result for curves of genus 2 without the need to fix the field of definition. The trade-off is that one must fix the discriminant of the real multiplication field contained in the complex multiplication of the Jacobian of the curve.

More specifically, the main results in the aforementioned source bound the discriminant of such quartic complex multiplication fields in terms of the minimal discriminant \footnote{The minimal discriminant represents the product of prime ideals at which the curve experiences bad reduction to the minimal exponent in the discriminant of an integer model. For more details, please refer to \cite[p. 2538]{HabeggerPazuki17}.} of the curves. 
\begin{thm}(\cite[Thm. 1.3]{HabeggerPazuki17})\label{thm:mainHP}
Fix a real quadratic field $F$. Let $C$ be a curve of genus $2$ defined over $k\subseteq\overline{\Q}$ such that its jacobian has CM by the maximal order of $K$, where $K$ is a cyclic quartic Galois CM field containing $F$. Then 
$$
\log \Delta_K\leq c(F)\left(1 +\frac{1}{[k:\mathbb{Q}]}\log \operatorname{N}(\Delta^0_{min}(C_k))\right), 
$$
where $c(F)$ is a constant only depending on $F$.
\end{thm}

In the following we will sometimes abbreviate the jacobian of $C$ has CM by the maximal order of $K$ by saying that $\Jac(C)$ has CM by $K$ or even shorter, by saying that $C$ has CM by $K$. 

\begin{cor}(\cite[Thm. 1.1]{HabeggerPazuki17})
Fix a real quadratic field $F$. Let $C$ be a curve of genus $2$ defined over $\overline{\Q}$ with good reduction everywhere and such that its jacobian has CM by $K$, where $K$ is a cyclic Galois CM field containing $F$. Then the discriminant $\Delta_K$ is bounded. In particular, there are only a finite number of such curves.
\end{cor}

We refer to \cite{GL12,LV15} for results bounding in the other direction the primes in the minimal discriminant of the curve in terms of the discriminant of the quartic field.

The idea behind the proof of Theorem \ref{thm:mainHP} is bounding the height of such an abelian variety from below and from above in terms of $\Delta_K$ and $\operatorname{N}(\Delta^0_{min}(C_k))$ and from there getting the desire bound on it. 

The first inequality is a consequence of a particular case of Colmez's Conjecture \cite{Colmez93} by Obus \cite{Obus13} and a result of Badzyan \cite[Thm. 1]{Badzyan10}: 

\begin{prop}(\cite[Prop. 4.3 (iv)]{HabeggerPazuki17})\label{easyhbound}
    Let $K$ be a CM field with $K/\mathbb{Q}$ cyclic of degree $4$ and let $F$ be the real quadratic subfield of $K$. Let $A$ be a simple abelian surface with endomorphism ring $\mathcal{O}_K$. Then, denoting by $h(A)$ the Faltings height of $A$, 
    \[
    h(A)\geq -c+\frac{\sqrt{5}}{20}\log \Delta_K 
    \]
    where $c$ is a constant that depends only on $F$. 
\end{prop}

By following the cited references we can conclude that we can take $c=\frac{\gamma_F}{2}+\frac{1}{4}\log \Delta_F+\log 2\pi+\gamma_\mathbb{Q}$ where $\gamma_\mathbb{Q}=0.566215...$ is the Euler constant and $\gamma_F$ is the Euler-Kronecker constant of $F$ \cite{Ihara06}. 

The laborious inequality is upper-bounding  the factors $h_i^{\infty}$ in the expression\footnote{{Notice the $1/4$ factor here coming from averaging in \cite[Thm. 4.5(i)]{HabeggerPazuki17} and forgotten in \cite[Eq. 6.1]{HabeggerPazuki17} with no repercussion on their final result.}}:

\begin{equation}\label{def:h}
h(\Jac(C))=h^0+{(}h_1^{\infty}+h_2^{\infty}+h_3^{\infty}+h_4^{\infty}{)/4}-\frac{4}{5}\log 2-\log\pi,
\end{equation}
where $h^{0}=\frac{1}{60[k:\mathbb{Q}]}\log \operatorname{N}(\Delta^0_{min}(C_k))$. More precisely, for each $\epsilon>0$, \cite[Eq. 6.3]{HabeggerPazuki17} gives
$$
h_i^{\infty}\leq \epsilon\log\Delta_K +c(\epsilon, F),
$$
where the constant $c(\epsilon,F)$ is not explicit.
In this paper we go through all the steps in the proof of \cite[Thm. 1.3]{HabeggerPazuki17}, we simplify some of them and we get explicit constants. The price to pay is having to assume the Generalized Riemann Hypothesis (GRH). 

We assume GRH at two different steps here: first, for the explicit subconvexity estimates of Chandee \cite{Chandee09} that we use in Lemma \ref{lemboundssubconv}. Second for lower bounds on class numbers \cite{Louboutin05} in Proposition \ref{deltalog} where it provides the best estimates to lower-bound the cardinal of the Class group. This second time, in principle we could avoid using the GRH by using the unconditional explicit result of Stark \cite{stark}, however, it does not provide a good enough exponent for our argument. 

Our main results are the following explicit versions of the main results in \cite{HabeggerPazuki17}.

\begin{thm}\label{thm:main}
Fix a real quadratic field $F$ and assume GRH. Let $C$ be a curve of genus $2$ defined over $k\subseteq\overline{\Q}$ such that its jacobian has CM by $K$, where $K$ is a cyclic Galois CM field containing $F$. Then 
$$
\Delta_K\leq\max\{e^{e^{64}}, (144000h_F^3R_F)^{64}, e^{10(\frac{1}{[k:\mathbb{Q}]}\log \operatorname{N}(\Delta^0_{min}(C_k))+\frac{\gamma_F}{2}+\frac{1}{4}\log \Delta_F + 8.4R_F + 1.4{5})}\}.
$$
\end{thm}

\begin{cor}\label{cor:main}
Fix a real quadratic field $F$ and assume GRH. Let $C$ be a curve of genus $2$ defined over $\overline{\Q}$ with good reduction everywhere and such that its jacobian has CM by $K$, where $K$ is a cyclic Galois CM field containing $F$. Then $$
\Delta_K\leq\max\{e^{e^{64}}, (144000h_F^3R_F)^{64}, e^{10(\frac{\gamma_F}{2}+\frac{1}{4}\log \Delta_F + 8.4R_F + 1.4{5})}\}.
$$
In particular, there are only a finite number of such curves.
\end{cor}

In addition of making the previous results explicit, our main contributions are the following: first in Subsection \ref{z12Disc} we give a explicit version of the \cite{HabeggerPazuki17} claim ``$|z_{12}|\geq H(Z)^{-4}$ and $H(Z)$ is polynomially bounded in $\Delta_K$ '' by providing the inequality $|z_{12}|\geq\frac{2}{3}|\Delta_K|^{-1/2}$. The idea is just tricky and smart manipulations of the reduced period matrix instead of using the Fundamental Theorem of Liouville (as in \cite{BombieriGubler}) and strong results of Tsimerman and Pila \cite{Pila}. We use this explicit version a second time in order to bound the contribution of the points of small height and we avoid having to use the non-explicit but deep equidistribution resuls of Zhang \cite[Cor. 3.3]{Zhang}. The second improvement is to control a term of the form $\sum_{[I] \in H} \left( \frac{|\Delta_K|^{1/2}}{N([I])} \right)^{1/2}$, where instead of using Michel and Venkatesh’s deep Theorem 1.1 \cite{MichelVenkatesh10}, we use explicit subconvexity results by Chandee \cite{Chandee09}. Again, it is worth noting that this approach necessitates an assumption of the GRH. The GRH assumption also arises in another context: in order to explicitly lower-bound the size of class group of $K$ in terms of its discriminant, we use the explicit version of the Brauer-Siegel Theorem by Louboutin \cite{Louboutin05}. Coming back to the previous term, in Questions \ref{question}, these estimates appear related to controlling on average (over almost the whole class group) the optimality of Minkowski bound on the integral ideals of minimal norms in a given class (for the quartic extensions $K$ considered) which is in our opinion of a strong independent interest. Finally, we fix the quadratic real field $F=\mathbb{Q}(\sqrt{5})$ and we obtain a bound for the discriminant of cyclic Galois quartic CM fields for which  genus 2 curves with CM by them  with potentially good reduction everywhere. Unfortunately, our bound becomes to big to run computations up to them, but we conjecture the corresponding finite list of such curves.

One could think of extending Theorem \ref{thm:main} by considering any quartic CM field $K$ and by considering larger genus. The condition $K$ cyclic Galois is used in several places, we stress this condition when needed, in particular, it is strongly used in the Colmez's formula for the Falting height in \ref{sec:heights}.
The approach should be completely reconsidered to include genus $3$ curves because of the hyperelliptic reduction that produces primes in the minimal discriminant not giving bad reduction \cite{LLLR}. Another possible generalisation could go in the direction of considering curves with CM by  general  orders of $K$: however, the maximality condition is strongly used in several places  as in the already mentioned Colmez's conjecture or on the bounds by Louboutin for class numbers.

The paper is structured as follows. In Section \ref{sec:setup} we present the set-up and the first definitions regarding abelian varieties with
complex multiplication as complex torus. In Section \ref{sec:lemmas} we give a explicit version of the needed technical lemmas bounding the entries of a reduced period matrix in terms of the discriminant of the CM quartic field. In Section \ref{sec:subconvex} we bound in a explicit way the average of the inverse of the minimal norm of elements of the class group via subvoncexity results à la Chandee. In this section we need to assume the GRH. 
In Section \ref{sec:boundh} we put all the precious results together to give a explicit upper-bound of the height  of a genus 2 jacobian with $CM$ by a cyclic quartic CM field with RM by a fixed real quadratic field $F$. This bound together with Proposition \ref{easyhbound} gives the proof of our main Theorem in Section \ref{sec:main}. Finally, in Section \ref{sec:comp} we fix the real field $F=\mathbb{Q}(\sqrt{5})$ and we make our bound explicit for this field. We perform numerical experiments up to a bound we can easily reach with our computers to list genus $2$ curves with good reduction everywhere and CM by a Galois cyclic CM field containing   $F$. We conjecture this list to be exhaustive.

\subsection*{Acknowledgement} We thank Philipp Habegger and Fabien Pazuki for insightful comments on a preliminary version of this article. We thank John Voight for ideas exchanges about the results in Subsection \ref{subsec:fields} to enumerate cyclic Galois CM quartic field containing a $\mathbb{Q}(\sqrt{5})$ and with bounded discriminant. We thank the Universit\"at G\"ottingen for the server capacity with which we conducted the computations in Subsection \ref{sec:comp}. The research of the first author is partially funded by the Swiss National Science foundation under the project number P2BSP2-181852. The research of the second author is partially funded by the IRGA project PointRatMod and by the project ANR JINVARIANT. The research of the third author is partially funded by the Melodia
ANR-20-CE40-0013 project and the 2023-08 Germaine de Sta\"el project.

\section{Set-up}\label{sec:setup}

We fix some notation as follows.

\begin{itemize}
    \item $k$ is a number field.
    \item $\Delta_k$ is the discriminant of a number field $k$.
    \item $R_k$ is the regulator of a number field $k$.
    \item $\zeta_k$ is the Dedekind Zeta-function of a field $k$.
    \item $F$ is a fixed real quadratic field.
    \item $K$ is a CM quartic Galois extension of $\Q$ containing $F$.
    \item $\Phi = (\phi_1,\phi_2)$ is a choice of CM type on $K$, and $
\Tr_{\Phi}(z) := \phi_1(z) + \phi_2(z) $ for all $z \in K$.
    \item $H$ will be a coset of $\Cl_K$ the class group of $K$.
    \item Fractional ideals of $K$ are denoted by $I$, and their class in $\Cl_K$ is denoted by $[I]$.
     \item $\mathbb{H}_2$ is the Siegel upper-half plane. General elements of $\mathbb{H}_2$ will be denoted as $Z = \begin{pmatrix} z_1 & z_{12} \\ z_{12} & z_2 \end{pmatrix}$ and for each index $i \in \{1,2,12\}$,
          $x_i = \Re(z_i)$ and $y_i = \Im(z_i)$.
      \item $\chi_{10}$ is the classical Siegel cusp form in  $\mathbb{H}_2$ equal to the product of the  squares of the 10 theta functions.
   
\end{itemize}

Furthermore, we define here the standard fundamental domain of $\mathbb{H}_2$ under the action of $\Sp_4(\mathbb{Z})$.

\begin{defi}\label{def:F2}
    For genus $2$, the \emph{fundamental domain} $\mathcal{F}_2$ is defined to be the set of $Z=X+iY \in \mathbb{H}_2$ for which
    \begin{itemize}
        \item[(S1)] the real part $X =  \begin{pmatrix} x_1 & x_{12} \\ x_{12} & x_2 \end{pmatrix} $ is reduced, i.e., $-\frac12 \leq x_i < \frac12$ $(i=1,2,12)$.
        \item[(S2)]the imaginary part $Y$ is $(\GL_2)-$reduced, i.e. $0\leq 2y_{12} \leq y_1 \leq y_2$, and
        \item[(S3)] $|\det M^* (Z)|\geq 1$ for all $M\in \Sp_4(\Z)$, where $M^*(Z)$ is defined by $$M^*(Z) = CZ+D \text{ for } M = \begin{pmatrix} A & B \\ C & D \end{pmatrix}\in \Sp_4(\Z)$$
    \end{itemize}
\end{defi}

Every point in $\mathbb{H}_2$ is $\Sp_4(\mathbb{Z})$-equivalent to a point in $\mathcal{F}_2$. This point is unique up to identifications of the boundaries of $\mathcal{F}_2$, see \cite[Lem. II.5.20]{Strengthesis}.

\begin{defi}[CM triple and associated Riemann form] 
    Let  $K$ be a CM quartic field. A \emph{CM triple} $(I,\xi,\Phi)$ on $K$ is given by a fractional ideal $I$ of $K$, a totally imaginary element $\xi$ of $K$ and a CM type $\Phi = (\phi_1,\phi_2)$ of $K$ such that $\xi \bar{I} I = \Dcal_{K/\Q}^{-1}$ the inverse of the different ideal and $\phi_1(\xi),\phi_2(\xi)$ have positive imaginary part.
       With such a triple, we define a $\R$-bilinear Riemann form\footnote{Our form $E$ equals $-E$ with the \cite{HabeggerPazuki17} notation, this is because we prefered $E$ to be left-sesquilinear insetad of not right-sesquilinear. However, with both notations the forms $H$ become equal since we use the definition  $H(z,w)=E(z,iw)+iE(z,w)$ instead of $H(z,w)=E(iz,w)+iE(z,w)$.} on $\C^2$ by
       \begin{equation}
    \label{eqdefE}
    E_{\Phi,\xi} (z,w) :=\sum_{j=1}^2 \phi_j(\xi) (z_j \overline{w_j} - \overline{z_j} w_j),
    \end{equation}
    and then for any $x, y \in K$,
    \begin{eqnarray}
    \label{eqpropE}
    E_{\Phi,\xi} (\Phi(x),\Phi(y)) & = & - \Tr_{K/\Q} (\xi \overline{x} y) = - 2 \Re \Tr_\Phi(\xi \overline{x} y) \\
    E_{\Phi,\xi} (\Phi(x),i\Phi(y)) & = & 2 \Im \Tr_\Phi(\xi \overline{x} y).
     \end{eqnarray}
     In other words, $E_{\Phi,\xi}$ is the imaginary part of the hermitian (definite positive form) $H_{\Phi,\xi}$ defined on $\C^2$ by 
\begin{eqnarray}
    \label{eqdefH}
    H_{\Phi,\xi}(z,w) & = & - 2 i \sum_{j=1}^2 \phi_j(\xi) \overline{z_j} w_j \\
    H_{\Phi,\xi}(\Phi(x),\Phi(y)) & = & (-2i) \Tr_\Phi(\xi \overline{x} y).
\end{eqnarray}
     By construction, the form  $E_{\Phi,\xi}$ is a symplectic form for the lattice $\Phi(I)$ in $\C^2$ with pfaffian 1, and thus $\C^2 / \Phi(I)$ is given a structure of principally polarised abelian surface over $\C$, with complex multiplication by $\Ocal_K$, denoted by $A_{I,\xi,\Phi}$.
\end{defi}

It is a classical and beautiful result by Shimura and Taniyama \cite{ST61}, known as the Main Theorem of the Complex Multiplication Theory, that implies that the previous description produce all the pricipally polarised abelian varieties with CM.

\begin{defi}[Reduced matrix associated to CM triple]
    For a CM triple $(I,\xi,\Phi)$, every choice of symplectic $\Z$-basis $B = (e_1, e_2,e_3,e_4)$ of $I$, i.e such that $$\Tr_{K/\Q} (\xi \overline{e_i} e_j) = \begin{cases} 0 \text{ if } |i-j| \neq 2\\1 \text{ if } j = i+2\\-1 \text{ if } j = i-2,\end{cases}$$ defines an element 
\begin{equation}
\label{eqdefZ}
Z_{I,\xi,\Phi,B}  = \begin{pmatrix} z_1 & z_{12} \\ z_{12} & z_2 \end{pmatrix} \in \mathbb{H}_2
\end{equation}
which is by definition the matrix of the vectors $\Phi(e_3),\Phi(e_4)$ in the $\C$-basis $(\Phi(e_1),\Phi(e_2))$. In other words, $z_1$, $z_{12}$ and $z_2$ are uniquely determined by
\begin{eqnarray*}
\Phi(e_3) & = & z_1 \Phi(e_1) + z_{12} \Phi(e_2) \\
\Phi(e_4) & = & z_{12} \Phi(e_1) + z_{2} \Phi(e_2) 
\end{eqnarray*}
In particular, $z_1,z_{12},z_2 \in K$. Changing bases amounts to the action of $\Sp_4(\Z)$ on $\mathbb{H}_2$, therefore a reduced choice of $Z$ (in the fundamental domain) corresponds to a choice of basis $B$ and we denote by  $Z_{I,\xi,\Phi, \textrm{red}}$ the reduced matrix (in $\Fcal_2$) thus obtained, dropping some indices if the notation if not ambiguous.
\end{defi}

\subsection{Faltings heights}\label{sec:heights}

We follow Faltings \cite{Falt83} to define the height of an abelian variety $A$ of dimension $g$ with Néron model $\mathcal{A}$ over $S=\operatorname{Spec}\mathcal{O}_k$ with zero section $\epsilon:\,S\rightarrow \mathcal{A}$.

\begin{defi} The stable Faltings height of $A$ is
$$
h(A):=\frac{1}{[k:\mathbb{Q}]\widehat{\operatorname{deg}}(\omega)}\text{ where }\omega=\epsilon^*\Omega^g_{\mathcal{A}/S}.
$$
\end{defi}

Colmez's conjecture \cite{Colmez93} gives a formula for the Faltings height of an abelian variety with complex
multiplication in terms of the derivative of certain Artin $L$-functions at $s = 0$. Colmez proved his conjecture for CM fields that are abelian extensions of $\mathbb{Q}$ under some ramification conditions. Obus \cite{Obus13} generalised the result by removing these conditions.

For more details see \cite[Sec. 4]{HabeggerPazuki17}. We directly state here the result we will use:

\begin{thm}[Cor. 4.6 + details as in Section 6 of \cite{HabeggerPazuki17}] Let $C$ be a genus $2$ curve defined over a number field $k$ and with CM by a Galois quartic CM field $K$. Then the Faltings height of its jacobian is decomposed as:
$$
h(\operatorname{Jac}(C))=h^0+{(}h_1^\infty+h_2^\infty+h_3^\infty+h_4^\infty{)/4}-\frac{4}{5}\log 2-\log \pi,
$$
where the finite part is
\begin{equation}\label{finitpart}
h^0=\frac{1}{60[k:\mathbb{Q}]}\log \operatorname{N}(\Delta^0_{min}(C))
\end{equation}
and the infinity parts are:
\begin{equation}\label{infinitypart}h^\infty_i = -\frac{1}{10 \#H} \sum_{[ I ]\in H} \log (|\chi_{10} (Z_{I,\textrm{red}})| \det(\Im Z_{I,\textrm{red}})^5),\end{equation}
where $Z_{I,\textrm{red}}$ is a reduced period matrix for a CM-triple $(I,\xi,\Phi)$ having the same CM-type than $\Jac(C)$ and  $H$ is a particular coset of $\Cl_K$. 
\end{thm}

We do not need to precise the definition of the coset $H$ until Proposition \ref{deltalog}, so we will postpone it. One important thing to remark is (\cite[Eq. 6.1]{HabeggerPazuki17}))

\begin{equation}
{(}h_1^\infty+h_2^\infty+h_3^\infty+h_4^\infty{)/4}= -\frac{1}{10[k:\mathbb{Q}]} \sum_{\sigma:\,k\rightarrow \mathbb{C}} \log (|\chi_{10} (Z_{\,^\sigma I,\textrm{red}})| \det(\Im Z_{\,^\sigma I,\textrm{red}})^5).
\end{equation}

\section{Bounds on the entries of the reduced period matrix}\label{sec:lemmas}

We fix the notation from the set-up for the following section unless said otherwise. In particular as $Z = X+iY= \begin{pmatrix} z_1 & z_{12} \\ z_{12} & z_2 \end{pmatrix}\in \mathbb{H}_2$ and the fundamental domain $\mathcal{F}_2$ is given as in Definition \ref{def:F2}. We discuss in these section some bounds for the entries of the matrices $X$, $Y$ and $Z$. 

\subsection{Bounds on the diagonal coefficients of the imaginary part}
We mainly follow here the ideas in \cite[Sec. 5.1]{Strengthesis}. 

\begin{lem}\label{lem:34}
\label{mink}
For any $Y\in \mathcal{F}_2$,  $y_1y_2\leq \frac43 \det Y$.
\end{lem}

\begin{proof}
For $Y$ $\SL_2$-reduced we have 
$\det Y = y_1 y_2 - y_{12}^2\geq y_1y_2-\frac{y_1y_2}{4} =\frac{3}{4}y_1y_2$ where the inequality comes from condition (S2) in Definition \ref{def:F2}.

\end{proof}

\begin{lem}\label{lem:det} For $Y\in\mathcal{F}_2$, $\det(Y)\geq 9/8$.
\end{lem}
\begin{proof} Condition (S2) in Definition \ref{def:F2} implies $y_1\leq y_2$. On the other hand, (S3) with $M=\operatorname{Id}$ implies $|z_1|\geq1$, so (S1) gives $y_1\geq \sqrt{3/4}$ and the result follows. 
\end{proof}

\subsection{Lower bounds on the off-diagonal coefficients}\label{z12Disc}

The goal of this subsection is to prove Proposition \ref{propminz12}, a completely explicit version of the bound on off-diagonal coefficients of the reduced matrices, a key improvement on \cite[Eq. (6.5)]{HabeggerPazuki17}. 

As it turns, this bound is (by a very short margin) enough to obtain the main result and has the advantage of being completely effective, as opposed to the same flavour results used in \cite{HabeggerPazuki17}: First,  the Fundamental Theorem of Liouville \cite{BombieriGubler} and a powerful result by Tsimerman and Pila \cite{Pila}, and secondly, the strong equidistribution results by Zhang \cite{Zhang}, see \cite[p. 2568-2569]{HabeggerPazuki17}. 

\begin{prop}
\label{propminz12}
For a reduced matrix $Z  = \begin{pmatrix} z_1 & z_{12} \\ z_{12} & z_2 \end{pmatrix} 
$ corresponding to a choice of CM triple $(I,\xi,\Phi)$, we have 
\[
|z_{12}| \geq {\frac{2}{3}} |\Delta_K|^{-1/2}.
\]
\end{prop}

Before proving the previous proposition, we will need to set some technical lemmas. 
Until the end of this section, we fix the CM triple $(I,\xi,\Phi)$ and $Z$ is the corresponding associated reduced matrix.

\begin{lem}
    With the notation above, we have 
    \begin{eqnarray}
\label{eqpropy1y2Tr}
\frac{y_2}{\det Y} & = & \frac{2}{i}  \Tr_\Phi(\xi \overline{e_1} e_1) = H_{\Phi,\xi}(\Phi(e_1),\Phi(e_1)), \\
\frac{y_1}{\det Y}&  = & \frac{2}{i}  \Tr_\Phi(\xi \overline{e_2} e_2) =H_{\Phi,\xi}(\Phi(e_2),\Phi(e_2)), \\
\frac{y_{12}}{\det Y} & = & \frac{-2}{i}  \Tr_\Phi(\xi \overline{e_1} e_2) = - H_{\Phi,\xi}(\Phi(e_1),\Phi(e_2)).
\end{eqnarray}
\end{lem}

Notice that all these type traces are purely imaginary because $\Phi(e_1),\Phi(e_2)$ are orthogonal for $E_{\Phi,\xi}$.

 \begin{proof}
     It is enough to prove that $Y^{-1}$ is the matrix of the hermitian form in the basis $(\Phi(e_1),\Phi(e_2))$ of $\C^2$, and this is a general fact on Riemann forms (see for instance the proof of Theorem 1.3 of Chapter VI in \cite{Debarre99}). 
 \end{proof}

\begin{lem}\label{lem:Tr} 
For any $a,b \in I$ such that $\Tr_{\Phi}(\xi \overline{a} b) \neq 0$ and is purely imaginary, we have 
$|\Tr_{\Phi}(\xi \overline{a} b)|  \geq \begin{cases}
|\Delta_K|^{-1/4}  &  \textrm{if } \Im \phi_1 (\xi \overline{a}b) \textrm{ and } \Im \phi_2(\xi \overline{a} b) \textrm{ have the same sign,} \\ 
 \frac{2|\Delta_K|^{-1/2}}{ H_{\Phi,\xi}(\Phi(a),\Phi(a)) +H_{\Phi,\xi}(\Phi(b),\Phi(b)) } & \textrm{otherwise}.
 \end{cases}
$

\end{lem}

\begin{proof}
Define $x = \xi \overline{a}b \in \xi \bar{I} I = (\Dcal_{K/\Q}^{-1})$ and $x' = x - \overline{x} \in  (\Dcal_{K/\Q}^{-1})$ as well which is a totally imaginary element of $K$. By hypothesis, $\Tr_\Phi(x)$ is purely imaginary so we can rewrite it as $\Tr_\Phi(x')/2$. In particular, our hypothesis implies that $x' \neq 0$. If $\Im \phi_1(x')$ and $\Im \phi_2(x')$ have the same sign, we can write 

\[
|\Tr_\Phi(x')| =  |\phi_1(x')| + |\phi_2(x')|   \geq  2 \sqrt{|\phi_1(x')| |\phi_2(x')|} \geq 2 |N_{K/\Q} (x')|^{1/4} \geq  2  |\Delta_K|^{-1/4},
\]
which proves the first case.

 If  $\Im \phi_1(x')$ and $\Im \phi_2(x')$ do not have the same sign, independently of the choice of type $\Phi$,  $\{\sigma \circ \phi_1, \sigma \circ \phi_2\}$ is $\{\phi_1,\overline{\phi_2}\}$ or $\{\phi_2,\overline{\phi_1}\}$. As a consequence, 
\[
|N_{K/\Q}(\Tr_\Phi(x'))| = (|\phi_1(x') + \phi_2(x')| |\phi_1(x') - \phi_2(x')|)^2 = | \phi_1(x')^2 - \phi_2(x')^2|^2.
\]
 As $x'$ is totally imaginary and nonzero by hypothesis, we cannot have $\phi_1(x') = \phi_2(x')$ so 
\[
|\Tr_\Phi(x')| = \frac{|N_{K/\Q}(\Tr_\Phi(x'))|^{1/2}}{(|\phi_1(x')| +|\phi_2(x')|)} \geq \frac{|\Delta_K|^{-1/2}}{(|\phi_1(x')| +|\phi_2(x')|)}. \]
Now, for $j \in \{1,2\}$ we can define the hermitian (positive semi-definite) 
form
\[
H_j (z,t) := (-i) \phi_j(\xi) \overline{z_j}t_j
\]
such that for all $a,b \in K$:
\[
H_j (\Phi(a),\Phi(b)) = (-i) \phi_j(\xi \overline{a} b).
\]
By the Cauchy-Schwarz inequality (still valid for positive hermitian forms), we have 
\[
|H_j(\Phi(a),\Phi(b))| \leq \sqrt{H_j(\Phi(a)) H_j(\Phi(b))},
\]
which gives
\begin{eqnarray*}
|\phi_j(x')| & \leq & 2 |\phi_j(x)| \leq 2 \sqrt{|\phi_j(\xi |a|^2 ) \phi_j(\xi |b|^2)} \\
& \leq &  (|\phi_j(\xi |a|^2)| + |\phi_j(\xi |b|^2 )|).
\end{eqnarray*}
So we obtain (as the conjugation does not change the absolute value, and the $\phi_j(\xi |a|^2)$ are always totally imaginary with positive imaginary part):
\begin{eqnarray*}
|\phi_1(x')|  + |\phi_2(x')|& \leq & |\Tr_\Phi(\xi |a|^2)| + |\Tr_\Phi(\xi |b|^2))| \\
 & \leq & \frac{1}{2}(H_{\Phi,\xi} (\Phi(a),\Phi(a)) + H_{\Phi,\xi} (\Phi(b),\Phi(b))). 
\end{eqnarray*}
By combining this with the previous inequality, the lemma is proved.
\end{proof}

Let $B = (e_1,e_2,e_3,e_4)$ be a symplectic basis. By Equation \eqref{eqpropE}, $\Tr_\Phi(\xi \overline{e_1} e_2)$ is purely imaginary so we can apply the Lemma to $a=e_1$ and $b=e_2$, {assuming this type trace is nonzero}. In the opposite sign case, we get by Equation \eqref{eqpropy1y2Tr} that
\[
|\Tr_\Phi (\xi \overline{e_1} e_2)| \geq \frac{2 |\Delta_K|^{-1/2}}{(y_2 + y_1)/\det Y},
\]
which gives by applying again Equation \eqref{eqpropy1y2Tr} 
\[
|y_{12}| \geq \frac{2 (\det Y)}{(y_1 + y_2) |\Delta_K|^{1/2}}.
\]
Recall that we are assuming $Z$ to be reduced. Then  Lemma \ref{lem:34} and the proof of Lemma \ref{lem:det} give 
\[
\det Y \geq \frac{3}{4} y_{1} \cdot y_{2}, \quad y_{2} \geq y_{1} \geq \sqrt{\frac{3}{4}}.
\]
Using these to provide crude bounds, we get 

\begin{equation}
\label{eqboundy12red}
|y_{12}| \geq \frac{3 \sqrt{3}}{4 |\Delta_K|^{1/2}} \geq |\Delta_K|^{-1/2}.
\end{equation}
In the case of same sign for both values of $\Im \phi_i(x)$, the bound predicted by the Lemma is better, 
 so \eqref{eqboundy12red} always holds. \\

We consider now the situation in which $y_{12} = 0$. In this case, as $\Tr_\Phi(\xi \overline{e_1} e_2)$ is purely imaginary with imaginary part $2 y_{12}/\det Y$:
\begin{itemize}
    \item $z_{12}=x_{12}$ is real, $Y$ is diagonal and $\det Y = y_1 y_2$.
    \item $\phi_2(\xi \overline{e_1} e_2) = - \phi_1 (\xi \overline{e_1} e_2)$ (definition of the trace) so $\xi \overline{e_1}e_2 \in \Q( \sqrt{\Delta_F}$) where $F$ is the totally real quartic field contained in $K$.
    \item $\Phi(e_1)$ and $\Phi(e_2)$ are orthogonal for $H_{\Phi,\xi}$ by \eqref{eqdefH}.
\end{itemize}

Recall that by \eqref{eqdefZ}, $\Phi(e_3) = z_1 \Phi(e_1) + z_{12} \Phi(e_2).$

Now, consider $x = \xi \overline{e_2} e_3$. By symplecticity hypothesis, $\Tr_\Phi(x)$ is again purely imaginary, and can be rewritten by orthogonality as 
\begin{eqnarray*}
\Tr_\Phi(\xi \overline{e_2} e_3) & = & \frac{i}{2} H_{\Phi,\xi}(\Phi(e_2),\Phi(e_3)) =  \frac{i z_{12}}{2} H_{\Phi,\xi}(\Phi(e_2),\Phi(e_2)) \\
& = & \frac{i z_{12} y_1}{\det Y} = \frac{i z_{12}}{y_2}.
\end{eqnarray*}

In particular, this type trace cannot be 0: as an abelian surface with CM by $K$ cyclic has necessarily a primitive CM-type and hence, it is not a product of elliptic curves.

We can therefore apply Lemma \ref{lem:Tr} to state that either 
\[
|\Tr_\Phi(\xi \overline{e_2} e_3)| \geq {|\Delta_K|^{-1/4}}
\]
or (opposite sign case)
\[
|\Tr_\Phi(\xi \overline{e_2} e_3)| \geq \frac{2 |\Delta_K|^{-1/2}}{H_{\Phi,\xi}(\Phi(e_2),\Phi(e_2)) + H_{\Phi,\xi}(\Phi(e_3),\Phi(e_3))}
\]
Using the orthogonality of $\Phi(e_1)$ and $\Phi(e_2)$ again, we have
\begin{eqnarray*}
H_{\Phi,\xi}(\Phi(e_2),\Phi(e_2)) + H_{\Phi,\xi}(\Phi(e_3),\Phi(e_3)) & = & \frac{y_1 + |z_1|^2 y_2 + z_{12}^2 y_1}{ \det Y} \\
& = & \frac{1}{y_2} + \frac{y_1^2 + x_1^2}{y_1} + \frac{z_{12}^2}{ y_2}.
\end{eqnarray*}
Assuming now that $Z$ is in the fundamental domain, 
$\sqrt{3}/2 \leq y_1 \leq y_2$, $\det Y \geq 1$, $|x_1| \leq 1/2$, $|z_{12}| \leq 1/2$. Using these inequalities, we obtain

\[
H_{\Phi,\xi}(\Phi(e_2),\Phi(e_2)) + H_{\Phi,\xi}(\Phi(e_3),\Phi(e_3))  \leq \frac{4}{3}y_2 + y_2 + \frac{1}{{3}}y_2 + \frac{1}{{3}}y_2  \leq 3 y_2
\]
hence
\[
|z_{12}| \geq {\frac{2}{3}} |\Delta_K|^{-1/2},
\]
which concludes the proof of Proposition \ref{propminz12}.

\begin{rem}
    The bound of Proposition \ref{propminz12} holds for all reduced $Z$ corresponding to a CM triple $(I,\xi,\Phi)$, but one might hope for a better (explicit) exponent than $-1/2$ on average over the whole Galois orbit of a CM abelian variety associated to the triple by using variants of this method.
\end{rem}

\subsection{Bounding the trace of the conjugated period matrices} We recall that  $(I,\xi,\Phi)$ is a fixed CM triple with  $Z$ as corresponding associated reduced matrix.
For an ideal class $\Ccal \in \Cl_K$, we denote by $N(\Ccal)$ the minimal norm of the integral ideals in $\Ccal$.

\begin{lem}(\cite[Lemma 3.6 (i)]{HabeggerPazuki17})
\label{lem36} 
Let $Z$ be a reduced period matrix. There exists a constant $c > 0$ which depends only on $g$ with the following property:
\begin{align*}
\Tr (\Im ( Z))	\leq c\left( \frac{|\Delta_K|^\frac12}{N([I]^{-1})} \right)^\frac1g.
\end{align*}
\end{lem}

Our explicit version is

\begin{lem}
\label{lem:lemma36HPexplicit}

With the same notations as in Lemma \ref{lem36} above, for $g=2$  one can take the constant $c=2/3$.

\end{lem}

\begin{proof}
We closely follow the proof of \cite[Lem. 3.6(i)]{HabeggerPazuki17}.
Define $Y := \Im(Z) = \begin{pmatrix} y_1 & y_{12} \\ y_{12} & y_2 \end{pmatrix}$. By construction of $Z$ (relative to a choice of symplectic $\Z$-basis $(e_1, e_2,e_3,e_4)$ of $I$), $Y^{-1}$ is the matrix of the hermitian form $H_{\Phi}$ in $\Phi(e_1),\Phi(e_2)$, define $y'_1$ and $y'_2$ as its diagonal elements.

For any $x \in I$ non zero, we have 
\begin{eqnarray*}
H_\Phi(\Phi(x),\Phi(x)) & = & 2 \sum_{j=1}^2 |\phi_j(\xi)| |\phi_j(x)|^2  \\
& \geq & 4 \sqrt{|\phi_1(\xi) \phi_2(\xi)|} |\phi_1(x) \phi_2(x)| \\
& \geq & 4 |N_{K/\Q}(\xi)|^{1/4} |N_{K/\Q}(x)|^{1/2} \\
& \geq & \frac{4|N_{K/\Q}(x)|^{1/2}}{N(I)^{1/2}|\Delta_K|^{1/4}}.
\end{eqnarray*}
Now, as $x \in I$, we can write $(x)  = I J$ for some integral ideal $J \in [I]^{-1}$, and we get by passing to the norm
\[
H_\Phi(\Phi(x),\Phi(x)) \geq  \frac{4 N(J)^{1/2}}{|\Delta_K|^{1/4}} \geq  \frac{4 N([I]^{-1})^{1/2}}{|\Delta_K|^{1/4}},
\]
which holds for all nonzero $x \in I$ so provides a lower bound for $y'_1$ and $y'_2$. Now, as $Z$ is reduced, $Y$ is Minkowski-reduced. Hence, $\det Y \geq \frac{3}{4} y_1 y_2$ by Proposition \ref{mink}. We can thus write 
$y'_1 \leq \frac{4}{3 y_1}$ and $y'_2 \leq \frac{4}{3 y_2}$. The lower bounds on $y'_1, y'_2$ obtained above thus transform into an upper bound on $y_1,y_2$ and the proposition follows.
\end{proof}

\section{Subconvexity bounds and the average of the inverses of the minimal ideals norms}\label{sec:subconvex}
In the previous section we bounded the trace of the imaginary part of a reduced period matrix by $\frac{2}{3}\left( \frac{|\Delta_K|^\frac12}{N([I]^{-1})} \right)^\frac12$ in Lemma \ref{lem:lemma36HPexplicit}. While we do not have individual  control on those terms, we can still bound them on average when we let the ideals $I$ vary on a coset in the class group $\Cl_K$. This gives an explicit version of \cite[Proposition 5.9 $(i)$]{HabeggerPazuki17}, which we state now.

\begin{prop}
\label{propsubconv}
Let $F$ be a totally real number field. There is an absolute constant $c$ and a constant $c'(F)$ such that for every imaginary quadratic extension $K/F$ (which is Galois quartic over $\Q$) of discriminant $|\Delta_K|^{1/4} \geq 98$ and every coset $H$ of a subgroup of $\Cl_K$, assuming GRH for the Hecke L-functions of all characters $\chi_K : \Cl_K \ra \C^* $, one has

\[
\frac{1}{\# H} \sum_{[I] \in H} \left( \frac{|\Delta_K|^{1/2}}{N([I])} \right)^{1/2}  \leq c \frac{|\Delta_K|^{1/2 - 1/32}}{\# H} + c'(F).
\]
Furthermore, we can take $c = 80$ and $c'(F)=\frac{5800 R_F \log |\Delta_K|}{|\Delta_K|^{1/32}} + 20 R_F\leq  1.4 \cdot 10^5 R_F$.

\end{prop}

\begin{rem}
This result (even without stating the explicit constants) is slightly stronger than \cite[Proposition 5.9 $(i)$]{HabeggerPazuki17}: indeed, our constant $c$ is actually absolute and the dependency in $F$ in the bound is only in the additive error term $c'(F)$. This is most likely due to the assumption of GRH (whose original goal was making the exponent gain explicit), and the authors do not know if one can attain such an explicit result without assuming GRH. Indeed, looking at unconditional subconvexity results in the literature, the existence of an absolute exponent gain on $|\Delta_K|$ (here $1/32$) is known but not made explicit and the fact that the constant before this power of $|\Delta_K|$ can be made absolute is not clear to us.
\end{rem}

The proof of this result has the same structure as in \cite{HabeggerPazuki17}:

\begin{itemize}
    \item[$(a)$] State a general subconvexity result on the critical line for Hecke L-functions $L(\chi,s)$ where $\chi : \Cl_K \ra \C^*$ is a character, which can then be extended for $\Re(s) \in [1/2,2]$ by the Phragmen-Lindelöf principle.
    \item[$(b)$] For any $x \geq 0$ large enough, bound 
    \begin{equation}
    \label{eqsumnormlessthanx}
    \frac{1}{\# H} \sum_{\substack{[I] \in H \\ N([I]) \leq x}} \left( \frac{|\Delta_K|^{1/2}}{N([I])} \right)^{1/2}
    \end{equation}
    thanks to weighted sums of integrals of $L(\chi,s)$ along vertical lines of the domain of convergence of the series.
    \item[$(c)$] Use residue theorems to shift these vertical lines near the critical lines, and obtain bounds on \eqref{eqsumnormlessthanx}.
    \item[$(d)$] Finally, use Minkowski bounds on the smallest integral ideal in a given class to find a proper $x$, and obtain the Proposition.
\end{itemize}

We will now proceed along those lines (notice at steps $(a)$ and $(c)$, the case $\chi = 1$ has to be treated slightly differently because $\zeta_K$ has a pole at $s=1$).

 For any number field $K$ of degree $d$ and any character $\chi: \Cl_K \ra \C^*$, the associated Hecke $L$-function is defined for $\Re(s) >1$ by 
    \[
    L(\chi,s) := \sum_I \frac{\chi(I)}{N(I)^s}
    \]
    
    where $I$ goes through nonzero integral ideals of $\Ocal_K$. This is a Hecke $L$-series associated to a Grossencharakter of modulus 1 (see \cite[paragraph VII.8]{Neukirch} for details on what follows).

Our first use of GRH is here. It produces completely explicit (and sufficiently strong) bounds on our $L$-functions on the critical line.

\begin{lem}
\label{lemboundssubconv}
For any CM quartic field $K$, any nontrivial $\chi : \Cl_K \ra \C^*$ and any $t \in \R$, assuming GRH for this $L$-function: 
\begin{equation}
\label{eqboundsubconvnontrivial}
|L(\chi,1/2+it)| \leq 263 (1+|t|)^{3/4} |\Delta_K|^{3/16}.
\end{equation}
For $\chi$ trivial, we have for any $t \in \R$, assuming GRH for $\zeta_K$:
\begin{equation}
\label{eqboundsubconvzetaK}
 |\zeta_K(1/2+it)| \leq 839 (1+|t|)^{3/4} |\Delta_K|^{3/16}.
 \end{equation}
\end{lem}

The proof of the previous lemma will use a general subconvexity result of Chandee \cite{Chandee09} on entire $L$-functions assuming GRH, for which we need some reminders. 

    For any $t \in \R$, the vertically shifted $L$-function 
\[
 L_t(\chi,s) := L(\chi,s+it) = \sum_I \frac{\chi(I) N(I)^{- i t}}{N(I)^s} = L(\chi_t,s) \textrm{ with  } \chi_t(I) = \chi(I) N(I)^{-it}
\]
is also a Hecke $L$-series associated to a Grossencharakter of modulus 1. The gamma factor of $L(\chi, \cdot)$ is defined by 
\[
 L_\infty(\chi,s) := \pi^{-ds/2} 2^{(1-s) r_2} \gamma\left( \frac{s}{2} \right)^{r_1} \gamma(s)^{r_2} = \pi^{-ds/2 - r_2/2}  \gamma\left( \frac{s}{2} \right)^{r_1 + r_2} \Gamma\left( \frac{s+1}{2} \right)^{r_2}
\]
where $r_1$ is the number of real embeddings and $r_2$ the number of pairs of complex embeddings of $K$ (the second equality coming from Legendre duplication formula). The completed $L$-series 
\[
 \Lambda(\chi,s) :=  |\Delta_K|^{s/2} L_\infty(\chi,s) L(\chi,s)
\]
has a meromorphic continuation to $\C$, satisfies the functional equation 
\[
 \Lambda(\chi,s) = \Lambda(\overline{\chi}, 1-s)
\]
and it is entire unless $\chi$ is trivial.

Following the definitions and notation of \cite{Chandee09}, the L-function $L(\chi,\cdot)$ thus has degree $d = [K:\Q]$, $k_j = 0$ for $j \in \{1,\cdots, r_1+r_2\}$ 
and $k_j=1$ for $j \in \{r_1+r_2+1,\cdots,d\}$ . The arithmetic conductor $q(L(\chi,\cdot))$ is $|\Delta_K|$, and the \emph{analytic conductor} $C(L(\chi,\cdot))$ is 
\[
 C(L(\chi,\cdot)) = \frac{q(L)}{\pi^{d}}  \prod_{j=1}^{d} \left| \frac{1}{4} + \frac{k_j}{2} \right| = \frac{ 3^{r_2}|\Delta_K|}{(4 \pi)^{d}}.
\]

For the vertically shifted $L$-function $L(\chi_t,s)$, $k_j = -it$ for $j \in \{1,\cdots, r_1+r_2\}$ and $k_j = 1 - it$ for $j \in \{r_1+r_2+1,\cdots, d\}$. The conductor of $L_t(\chi,s)$ is then
\begin{equation*}
 C_t := C(L(\chi_t, \cdot))  = \frac{|\Delta_K|}{(4 \pi)^{d}} \sqrt{(1 + 4t^2)^{r_1+r_2} (9 + 4 t^2)^{r_2}}. 
\end{equation*}

We now apply these equalities in our case: $K$ is CM of degree 4, so $d= 4$, $r_2 = 2$ and $r_1 = 0$. We obtain that for all $t \in \R$,
\begin{equation*}
 \frac{C(L(\chi_t,\cdot))}{(1 + |t|)^4|\Delta_K|} \in [10^{-4}, 7 \cdot 10^{-4}].
\end{equation*}

   Furthermore, remark that if we define the coefficients $a(n)$ by the Dirichlet series
\[
 - \frac{L'(\chi,s)}{L(\chi,s)} = \sum_{n \geq 1} \frac{a(n)}{n^s},
\]
using the Euler product, $ |a(n)| \leq d \Lambda(n)$ where $\Lambda$ is the von Mangoldt function.

We now recall the main result of \cite{Chandee09}.

\begin{prop}[\cite{Chandee09}, Theorem 2.1]
	\label{thmChandee}
	Assume GRH. If $L(f,s)$ is an L-function of analytic conductor $C_L$  and degree $d$ and $a(n)$ is the sequence of coefficients of $-L'/L$, for any parameters $\lambda \geq 1/2$ and $x$ such that $\log x \geq \max (2, 2 \lambda)$:
	\begin{eqnarray*}
	\log |L(f,1/2)| & \leq & \Re \left( \sum_{n \leq x} \frac{a(n)}{n^{1/2 + \lambda/\log x}} \frac{\log (x/n)}{\log x \log n}\right) \\
	& + & \left( \frac{1 + \lambda}{2}\right) \frac{\log C_L}{\log x} + \frac{(\lambda^2 + \lambda)d}{(\log x)^2} + \frac{d e^{- \lambda}}{x^{1/2} \log (x)^2}.
	\end{eqnarray*}
\end{prop}

\begin{proof} (of Lemma \ref{lemboundssubconv})
Proposition \ref{thmChandee}  applies to all $L(\chi_t,\cdot)$ for $\chi$ nontrivial (as these L-functions are admissible under Chandee's definition), and that paper actually provides an explicit (and sufficient) bound when $\log \log C(L(\chi_t,\cdot)) \geq 10$, but we need a result for all possible analytic conductors. We thus follow the outline \cite[Proof of Corollary 1.1]{Chandee09} . 

Trying to find acceptable values for the first term and the factor before $\log C_L$ in the right-hand side (the latter has to be $<1/4$ to attain subconvexity), we settle for $\lambda = 0.5$ and $x = e^4$ and apply the outline the same arguments, which gives \eqref{eqboundsubconvnontrivial} after numerical computation.

Regarding $\zeta_K$, we notice $\zeta_K/\zeta_\Q$ is an entire L-function of degree 3 admissible for this method, and we obtain with the same choice of $\lambda = 0.5$ and $x = e^4$ inside Proposition \ref{thmChandee} that for all $t \in \R$:
	\[
	\left|\frac{\zeta_K(1/2+it)}{\zeta(1/2+it)} \right| \leq 181 \left|\frac{|\Delta_K| (|1+2 i t| |3+ 2it|^2}{(4 \pi)^3} \right|^{3/16} \leq 60 (1+|t|)^{9/16} |\Delta_K|^{3/16}.
	\]
Next, following \cite{PlattTrudgian15} (and checking for small $|t|$ numerically), we have that for all $t \in \R$:
	\[
	|\zeta(1/2 + it) | \leq  13 (1+|t|)^{3/16}.
	\]
	Combining both bounds, we obtain \eqref{eqboundsubconvzetaK}.
\end{proof}

We can now explain the link between \eqref{eqsumnormlessthanx} and the studied $L$-functions.

Define the function $f(y) := y^{-1/2} e^{-y}$ on $(0,+ \infty)$ so that its Mellin transform is $\Gamma(s-1/2)$.

\begin{defi}\label{def:S}
    For any coset $H = h H_0$ of a subgroup of $\Cl_K$, and any $x>0$, we define
    \[
S_H(x) := \sum_{\substack{I \textrm{ ideal}\\ [I] \in H \\ N(I) \leq x}} f \left( \frac{N(I)}{x}\right).
    \]
    We define similarly for any character $\chi:\Cl_K \ra \C^*$ and any $x>0$
    \[
S(x,\chi) := \sum_{I \textrm{ ideal}} \chi(I) f \left( \frac{N(I)}{x} \right). 
    \]
\end{defi}

These sums will allow us to compute the average sum (\ref{eqsumnormlessthanx}), thanks to the following lemma.

\begin{lem}
\label{lemlinktoSxX}
With the same notation as in Definition \ref{def:S}, we have for every character $\chi$ and $x>0$:
\[
S(x,\chi) = \frac{1}{2 i \pi} \int_{\Re(s) = \sigma} \Gamma \left(s - \frac{1}{2} \right) L(s,\chi) x^s ds ,
\]
and for any $x >0$,
\[
\sum_{\substack{I \textrm{ ideal}\\ [I] \in H \\ N(I) \leq x}} \left(\frac{x}{N(I)} \right)^{1/2} \leq e S_H(x) \leq  \frac{e}{[\Cl_K : H_0]} \sum_{\substack{\chi \\ \chi_{|H_0} \equiv 1}} \chi(h)^{-1}  S(x,\chi).\]

\end{lem}

\begin{proof}
    First, discrete Fourier transform on the abelian group $\Cl_K$ directly gives that 
    \[
  \sum_{\substack{I \\ [I] \in H}} f \left( \frac{N(I)}{x}\right) \geq  \sum_{\substack{I \\ [I] \in H \\ N(I) \leq x}} f \left( \frac{N(I)}{x}\right) \geq \frac{1}{e}\sum_{\substack{I \\ [I] \in H \\ N(I) \leq x}} \left(\frac{x}{N(I)} \right)^{1/2}.
    \]
Next, by absolute convergence of the series defining $S(x,\chi)$ ($f$ decays rapidly), by $\Gamma$ decaying quickly on vertical lines, and by the inverse Mellin formula we can write for $\sigma>1$:
\begin{eqnarray*}
S(x,\chi) & = & \sum_I \chi(I)\cdot \frac{1}{2 i \pi} \int_{\Re(s) = \sigma} \Gamma \left(s - \frac{1}{2} \right)(N(I)/x)^{-s} ds \\
 & = &  \frac{1}{2 i \pi} \int_{\Re(s) = \sigma} \Gamma \left(s - \frac{1}{2} \right)\sum_I \frac{\chi(I)}{N(I)^s} x^s ds \\
 & = & \frac{1}{2 i \pi} \int_{\Re(s) = \sigma}  \Gamma \left(s - \frac{1}{2} \right) L(s,\chi) x^s ds.
\end{eqnarray*}
    
\end{proof}

We can now explain how to bound the sums $S(x,\chi)$.

\begin{prop}
\label{lemboundSxX}
	For any  $x = \varepsilon |\Delta_K|^{1/2}$ with $\varepsilon \in (0,1]$, we have 
	\begin{eqnarray*}
		|S(x,\chi)| & \leq & 163 \varepsilon^{1/2 + 1/12} |\Delta_K|^{1/2-1/32} \quad (\chi \neq 1) \\
		|S(x,1)| & \leq & 393 \varepsilon^{1/2 + 1/12} |\Delta_K|^{1/2-1/32} + \frac{\pi^{5/2}}{2} \varepsilon  h_K R_K.
	\end{eqnarray*}
\end{prop}

\begin{proof}
    First, we shift the vertical integrals appearing to some $\sigma \in (1/2,1)$ (to be fixed later). For $\chi$ nontrivial, there is no pole between the two vertical lines and for $\chi$ trivial, there is a pole at $s=1$ with residue $\Gamma(1/2) \kappa_K x$ (where $\kappa_K$ is the residue of $\zeta_K$ at $s=1$), so we get
    \[
S(x,\chi) = \frac{1}{2 i \pi}\int_{\Re(s)=\sigma} \Gamma \left( s - \frac{1}{2} \right) L(s, \chi) x^s ds + \mathbb{1}_{\chi \equiv 1} \Gamma(1/2)  \kappa_K x
    \]
    
    Define $\ell$ the affine function interpolating 1 at $1/2$ and 0 at 2 (i.e. $\ell(\sigma) = - 2/3 \sigma + 4/3$). 

Assume first that $\chi$ is nontrivial. As $L(\chi,\cdot)$ is then holomorphic, by the Phragmen-Lindelöf principle \cite[Theorem 5.53]{IwaniecKowalski}, for every $\sigma \in [1/2,2]$:
\[
 |L(\sigma + i t, \chi)|  \leq  263^{\ell(\sigma)} (1+|t|)^{3\ell(\sigma)/16 } \zeta(2)^{4(1-\ell(\sigma))} |\Delta_K|^{3\ell(\sigma) /16}.
\]
Define $c_\sigma := 263^{\ell(\sigma)} \zeta(2)^{4 (1 - \ell(\sigma))}$ (as $\sigma$ will be close to $1/2$, it will be close to but slightly inferior to 263).  We use for the $\Gamma$ function the Lerch bound \cite[p. 15]{Godefroy}
\begin{equation}
\label{eqbound}
|\Gamma(\sigma + it) | \leq \Gamma(1+ \sigma) \frac{\sqrt{1+t^2}}{\sqrt{\sigma^2 + t^2}} \sqrt{\frac{t \pi}{\sinh (t \pi)}}. 
\end{equation}

With this and writing $x = \varepsilon |\Delta_K|^{1/2}$ with $\varepsilon \leq 1$, we obtain

\begin{equation*}
\label{eqboundSxX}
	|S(x,\chi)|  \leq  \frac{c_\sigma \Gamma(1/2+\sigma)}{2 \pi} \varepsilon^{\sigma} |\Delta_K|^{\sigma/2 + 3 \ell(\sigma)/16} \int_{\R} (1+|t|)^{3/16} \sqrt{\frac{(1+t^2) \pi t}{((\sigma-1/2)^2 + t^2)\sinh(\pi t)}} dt.
\end{equation*}

Now, we need the exponent of $|\Delta_K|$ to be less than $1/2 - 1/32$. For this, an acceptable value of $\sigma$ is 
\[
\sigma_1 := \frac{1}{2} + \frac{1}{12},
\] 
so $\ell(\sigma_1) = 17/18$ and $c_{\sigma_1} \simeq 216$. For this value, we then get by numerical computation 

\begin{equation*}
|S(x,\chi)| \leq 163 \varepsilon^{\sigma_1} |\Delta_K|^{1/2 - 1/32}.
\end{equation*}

 Assume now that $\chi$ is trivial. We can apply the Phragmen-Lindelöf bound to $\zeta_K(s)/\zeta_\Q(s)$ with the bounds in the proof of Lemma \ref{lemboundssubconv}, and separately to $\zeta_\Q$. For the latter, we make use of Theorem 1 of \cite{Yang23} (for $k=4$) which implies (after real analysis and completing it for small values of $t$) that for $\sigma' = 5/7$ and all $t \in \R$, 
\[
|\zeta(\sigma' + i t)|\leq 6.4 (1+|t|)^{3/16 \ell (\sigma')}.
\]
Using Phragmen-Lindelöf interpolation principle now between $1/2$ and $\sigma'$, we get the bound 
\[
|\zeta(\sigma_1 + i t)|\leq 10 (1+|t|)^{3/16\ell (\sigma_1)  }.
\]
Applying now the same method as above (which boilds down to replacing $c_{\sigma_1} \simeq 216$ by 520 after computations and evaluating the same terms otherwise), we get

\begin{equation*}
	|S(x,1)| \leq   393 \varepsilon^{\sigma_1} |\Delta_K|^{1/2-1/32} + \kappa_K \Gamma(1/2) x.
\end{equation*}

As the residue $\kappa_K$ of $\zeta_K$ at $s=1$ is $h_K R_K \pi^{2}/2$ and $\Gamma(1/2) = \sqrt{\pi}$, we obtain the result.

\end{proof}
\begin{proof}[Proof of Proposition \ref{propsubconv}]
We now give the final steps of the proof. First, for any $x = \varepsilon |\Delta_K|^{1/2}$ with $\varepsilon \in (0,1)$, combining Proposition \ref{lemboundSxX} and \ref{lemlinktoSxX}, we obtain 
\begin{eqnarray*}
\frac{1}{|H|} \sum_{\substack{[I] \in H \\ N([I]) \leq x}} \left( \frac{x}{N(I)} \right)^{1/2} & \leq & \frac{e}{h_K} \sum_{\substack{\chi \\ \chi_{|H_0} \equiv 1}} |S(x,\chi)| \\
& \leq &  \frac{e}{h_K} 163 [\Cl_K : H_0] \varepsilon^{1/2+1/12} |\Delta_K|^{1/2-1/32} \\
& + & \frac{e}{h_K} (230 \varepsilon^{1/2+1/12} |\Delta_K|^{1/2-1/32} + \frac{\pi^{5/2} \varepsilon h_K R_K}{2})  \\
& \leq & C(H) \varepsilon^{1/2 + 1/12} + \varepsilon \frac{e \pi^{5/2} R_K}{2}. \\
C(H) & = & 163 e \frac{|\Delta_K|^{1/2-1/32}}{|H|} + 230 e \frac{|\Delta_K|^{1/2-1/32}}{h_K}.
\end{eqnarray*}

Now, for any $\varepsilon>0$, for each class $[I]$ of $H$, either the class has an integral ideal with norm less than $\varepsilon |\Delta_K|^{1/2}$ (hence it is counted in the previous sum), or it does not have it and then $|\Delta_K|^{1/2}/N([I]) \leq \varepsilon^{-1}$. 
By an immediate minimisation of the function $C(H) \varepsilon^{1/12} + \varepsilon^{-1/2}$, we obtain after bounding roughly 

\begin{eqnarray*}
\frac{1}{|H|} \sum_{[I] \in H}  \left( \frac{|\Delta_K|^{1/2}}{N([I])} \right)^{1/2} & \leq & 80 \frac{|\Delta_K|^{1/2-1/32}}{|H|} +  108 \frac{|\Delta_K|^{1/2-1/32}}{h_K} + 10 R_K.
\end{eqnarray*}

The constant $c$ in Proposition \ref{propsubconv} will thus be given by 80, and we will now bound the two last terms on the right to compute $c'(F)$.

First, considering the residue of $\zeta_K$ as $s=1$ again \cite[Thm. XIII.3.2]{MR1282723}, we have (for $K\neq\mathbb{Q}(\zeta_5)$):
\[
\kappa_K = \frac{2 \pi^2 h_K R_K}{|\Delta_K|^{1/2}}. 
\]
That is bounded, according to \cite[Theorem 4 (2)]{Louboutin05} (assuming $|\Delta_K|^{1/4} \geq 98$) by
\[
\kappa_K = \frac{2 \pi^2 h_K R_K}{|\Delta_K|^{1/2}} \geq \frac{2}{e \log |\Delta_K|}.
\]
This gives 
\[
 108 \frac{|\Delta_K|^{1/2-1/32}}{h_K} \leq 5800 \frac{R_K \log |\Delta_K|}{|\Delta_K|^{1/32}} \leq 1.4 \cdot 10^5 R_F.
\]
Where the last inequality comes from \cite[Proposition 4.16]{WashingtonCyclotomic} and that the maximum of the function $\frac{\log x}{x^{1/32}}$ is attained at $x=e^{32}$. This concludes the proof of Proposition \ref{propsubconv}.
\end{proof}

\begin{rem} We could have use at the end of the proof the unconditional explicit result by Stark in \cite[Thm. 1']{stark} instead of \cite[Theorem 4 (2)]{Louboutin05}. However, for big discriminants this would give a bound of the form $\kappa_K\geq\frac{c}{|\Delta_K|^{1/4}}$, which is not sufficiently fine for our argument at the end of Section \ref{sec:boundh} to prove Theorem \ref{thm:main}. 
\end{rem}

\begin{rem}
The huge value of $c'(F)$ might be improved by refining the subconvexity estimates for the Dedekind zeta function $\zeta_K$ on the critical strip, in terms of $\kappa_K$.

Additionally, remark that besides from the quite trivial dichotomy at the end of the proof, one never actually counts the ideals of minimal norm in classes, but all ideals of bounded norm. This suggests that the estimates given here might be much too large compared to what one should expect but this natural idea is surprisingly difficult to exploit in a significant way. In particular, it is not clear to the authors whether the bound given for $c'(F)$ in Proposition \ref{propsubconv} could have its last term constant instead of linear in $R_F$. In other words, one can raise the following (slightly weaker) question: 

\begin{question}\label{question} Is  
\[
\frac{1}{h_K} \sum_{\Ccal \in \Cl_K} \left( \frac{|\Delta_K|^{1/2}}{N(\Ccal)} \right)^{1/2}
\]
uniformly bounded when $K$ goes through all cyclic quartic CM fields?
\end{question}

What suggests that this might be the case is that counting the number of ideals of norm less than $|\Delta_K|^{1/2}$ (divided by $h_K$) by similar methods also leads to a main term linear in $R_F$ when $|\Delta_K|^{1/2-1/32}/h_K$ is small (which it is asymptotically). In other words, one can prove that asymptotically the average value of $|\Delta_K|^{1/4}/N(I)^{1/2}$ amongst all ideals $I$ such that $N(I) \leq |\Delta_K|^{1/2}$ is bounded by an absolute constant when  $|\Delta_K|^{1/2-1/32}/h_K$ is small. 

Furthermore, let us consider the Hilbert modular surface $S_F$ associated to principally polarised abelian surfaces with multiplication by $\Ocal_F$ and assume for simplicity that $h_F=1$. By \cite[Lemma 3.6 $(ii)$]{HabeggerPazuki17}, denoting by $\mu(\infty,\tau)$ the inverse of the distance of the point $\tau$ in the fundamental domain of $S_F$ (closest to $\infty$) associated to $A_{I,\xi,\Phi}$ to the cusp $\infty$ of $S_F$, we know that $\mu(\infty,\tau)$ is less than an absolute constant times $|\Delta_K|^{1/2}/N ([I])$. On another hand, assuming subconvexity bounds 
 (which we can under GRH, or using \cite[Theorem 1.2]{MichelVenkatesh10}), Zhang's equidistribution theorem holds \cite[Theorem 3.4]{Zhang} so the average value of $\sqrt{\mu(\infty,\tau)}$ over the Galois orbit of $A_{I,\xi,\Phi}$ has limit 
 \begin{equation*}
\frac{1}{\vol(D_F)} \int_{D_F} \sqrt{\mu(\infty,z)} \omega(z)
 \end{equation*}
where $D_F$ is a fundamental domain of $S_F$ in the sphere of influence of $\infty$ and $\omega$ the canonical measure on $D_F$, e.g, given on the square $\Hcal^2$ of the Poincaré upper-half plane with complex coordinates $z_1=(x_1,y_1)$ and $z_2=(x_2,y_2)$ by 
\[
\omega(z) = \frac{\dd x_1 \wedge \dd y_1 \wedge \dd x_2 \wedge \dd y_2}{4 \pi^2 y_1^2 y_2^2}.
\].
    
Now, Siegel's theorem \cite[Theorem IV.1.1]{vanderGeerHilbert} states that $\vol(D_F) = 2 \zeta_F(-1) \gg |\Delta_F|^{3/2}$, and one can prove that if $(z_1,z_2) \in D_F$, $y_1 y_2 \geq \frac{1}{4|\Delta_F|}$ by \cite[Lemma I.2.2, notice the typo interverting $s$ and $s^{-1}$ in the conclusion of the proof]{vanderGeerHilbert} and $(x_1,x_2)$ belongs to a fundamental domain of $\R^2 / \Phi(\Ocal_F)$ which is of volume $\sqrt{|\Delta_F|}$ up to a constant.

Combining these arguments, one can prove (after the change of variables $(t,u) = (\sqrt{y_1y_2},  \scriptstyle{\sqrt{y_1/y_2}}$) that 
\[
\frac{1}{\vol(D_F)} \int_{D_F} \mu(\infty,\tau) \omega (z) \ll \frac{1}{|\Delta_F|} \int_{e^{-R_F/2}}^{e^{R_F}/2} \frac{\dd u}{u} \int_{t \geq 1/\sqrt{|\Delta_F|}} \frac{\dd t}{t^2} \ll \frac{R_F \sqrt{|\Delta_F|}}{|\Delta_F|}
 \]

Unfortunately, the authors did not achieve to reach a definitive conclusion in one direction or the other regarding this situation, which could have very interesting consequences regarding this problem.

\end{rem}

\subsection{From a polynomial term to a logarithmic one}

The goal of this section is to explicitly upper bound $\frac{| \Delta_K |^{1/2}}{\#H}$ in terms of $\log \Delta_K$.

\begin{lem}\label{lem:delta}
For any positive integer $N$ we denote by $\delta(N)$ the number of different prime numbers dividing $N$. We then have for all $N \geq 10$:

\[
2^{\delta(N)} \leq N^{\frac{1}{\log \log N}}.
\]
\end{lem}

\begin{proof}
This is a straightforward consequence of \cite[Thm. 11]{Robin83}, and the fact that $\log(2) \times 1.39 \sim 0.96 <1$.
\end{proof}

\begin{prop}\label{deltalog}

Assuming GRH, for any CM triple $(\xi,I,\Phi)$ on $K$, the Galois orbit of the associated abelian variety is of cardinality $|H|$ with the coset $H$ of a subgroup of $\Cl_K$ such that
\[
\frac{| \Delta_K |^{1/2}}{\#H}\leq 215 h_F^3 R_F \log |\Delta_K||\Delta_K|^{1/\log \log |\Delta_K|}
\]
for $\Delta_K> 9.3 \cdot 10^7$. 
\end{prop}

\begin{proof}
First, by \cite[Theorem 3.9 and Lemma 3.10]{HabeggerPazuki17}, the Galois orbit of the abelian variety is in bijection with $H_0 = N_{\Phi}(\Cl_K)$ with $N_{\Phi}$ the type norm associated to $\Phi$ (here $K$ is Galois), and this image is itself of order $\#H = \#H_0 \geq \frac{h_K}{\#\operatorname{Cl}_K[2] h_F}$.

Now using again \cite[Thm. 4 (2)]{Louboutin05} under GRH, we have $h_K R_K \geq \frac{| \Delta_K|^{1/2}}{\pi^2 e \log \Delta_K}$ and $R_K\leq2R_F$ by \cite[Prop. 4.16]{WashingtonCyclotomic}.
On the other hand \cite[Prop. 6.3(2)]{Zhang} implies $\#\operatorname{Cl}_K[2]\leq2^{2\cdot\operatorname{rk}_2(\Cl_F[2])+2+\delta}$ where $\delta$ is the number of prime ideals in $\Ocal_F$ ramified in $\Ocal_K$.  Notice that then $\delta=\delta(|\Delta_{K/F}|)$, so Lemma \ref{lem:delta} implies $2^\delta\leq |\Delta_{K/F}|^{\frac{1}{\log\log |\Delta_{K/F}|}}\leq |\Delta_{K}|^{\frac{1}{\log\log |\Delta_{K}|}}$ and the result follows. 

\end{proof}

\section{An explicit upper-bound for the infinity part of the Faltings height}\label{sec:boundh}
In this section we explicitly bound each of the infinity parts of the Faltings height of $h(\operatorname{Jac}(C))$ for $C$ a genus 2 curve with CM by $K$ a Galois cyclic quartic CM field, in terms of the discriminant $\Delta_K$. We recall that each inifinity part is given by Equation (\ref{infinitypart}):

\begin{align*}
h_i^\infty = -\frac1{10 \# H} \sum_{[I]\in H} \log \left(|\chi_{10}(Z_{I, \textrm{red}})| \det (\Im Z_{I, \textrm{red}})^5\right).
\end{align*}
Here, we are in the case of $\Jac(C)$, so our abelian variety is indecomposable, hence the off-diagonal entries of $Z_{I, \textrm{red}}$ are non-zero and $\chi_{10}(Z_{I, \textrm{red}})$ is non-zero.

From Proposition 5.6 of \cite{HabeggerPazuki17}, we get that
\begin{align*}
|\chi_{10} (Z)| \geq 8 \cdot 10^{-5} \min \{ 1, \pi |z_{12}|\}^2 e^{-2\pi \Tr \Im Z_{I, \textrm{red}}}.
\end{align*}

Hence 

\begin{align*}
h_i^\infty =& -\frac1{10 \# H} \sum_{[I]\in H} \log (|\chi_{10}(Z_{I, \textrm{red}})| \det (\Im Z_{I, \textrm{red}}))^5= \\
& \frac1{10 \# H} \sum_{[I]\in H} (-\log |\chi_{10}(Z_{I, \textrm{red}})| - 5\log \det (\Im Z_{I, \textrm{red}}))\leq \\
& \frac1{10 \# H} \sum_{[I]\in H} (-\log (8 \cdot 10^{-5} \min \{ 1, \pi |z_{12}|\}^2 e^{-2\pi \Tr (\Im Z_{I, \textrm{red}})}) - 5\log \det (\Im Z_{I, \textrm{red}}))= \\
& - \frac{\log (8\cdot 10^{-5})}{10} +\frac1{10 \# H} \sum_{[I]\in H} (\log (\min \{ 1, \pi |z_{12}|\}^{-2}) + 2\pi \Tr (\Im Z_{I, \textrm{red}}) - 5\log \det (\Im Z_{I, \textrm{red}})).
\end{align*}

In the next step of the proof in \cite{HabeggerPazuki17}, they  bound  $\det(\Im (Z_{I,\textrm{red}}))$ from below by a non-explicit constant $c_2$. Here we use Lemma \ref{lem:det} and we get:

\begin{align*}
h_i^\infty & \leq - \frac{\log (8\cdot 10^{-5})}{10}- \frac12 \log \frac98+\frac1{10 \# H} \sum_{[I]\in H} (\log (\min \{ 1, \pi |z_{12}|\}^{-2}) + 2\pi \Tr (\Im (Z_{I, \textrm{red}}))).
\end{align*}

Now we use our explicit version of \cite[Lem. 3.6(i)]{HabeggerPazuki17}, i.e. Lemma \ref{lem:lemma36HPexplicit} to bound the trace of the imaginary part of the reduced period matrix in order to get:

\begin{equation*}
h_i^\infty \leq - \frac{\log (8\cdot 10^{-5})}{10}- \frac12 \log \frac98+\frac{\sum_{[I]\in H} \left(\log (\min \{ 1, \pi |z_{12}|\}^{-2}) + \frac{4\pi}{3}\left( \frac{|\Delta_K|^\frac12}{N([I^{-1}])}\right)^{1/2} \right)}{{10 \# H}}.
\end{equation*}

Now, we use our Proposition \ref{propminz12} to explicitly bound $|z_{12}|$ in terms of $|\Delta_K|$: 

\begin{equation*}
h_i^\infty \leq - \frac{\log (8\cdot 10^{-5})}{10}- \frac12 \log \frac98+ \frac{\log \frac{{9}|\Delta_K|}{{4}\pi^2}}{10}+\frac{\sum_{[I]\in H} \frac{2\pi}{3}\left( \frac{|\Delta_K|^\frac12}{N([I^{-1}])}\right)^{1/2}}{{5 \# H}}.
\end{equation*}

Next, we apply our subconvexity bound in Proposition \ref{propsubconv} (so we need to assume here $|\Delta|_K\geq 9.3\times 10^7$) to control the last term:

\begin{equation*}
h_i^\infty \leq - \frac{\log (8\cdot 10^{-5})}{10}- \frac12 \log \frac98+ \frac{\log \frac{{9}|\Delta_K|}{{4}\pi^2}}{10}+\frac{2\pi}{15}\left(80\frac{|\Delta_K|^{1/2 - 1/32}}{\# H}+\frac{5800 R_F \log |\Delta_K|}{|\Delta_K|^{1/32}} + 20 R_F\right).
\end{equation*}

Now, Proposition \ref{deltalog} gives:

\begin{equation*}
\begin{split}
h_i^\infty \leq &- \frac{\log (8\cdot 10^{-5})}{10}- \frac12 \log \frac98+ \frac{\log\frac{{9}|\Delta_K|}{{4}\pi^2}}{10}\\
&+\frac{2\pi}{15}\left(80\cdot215 h_F^3 R_F \frac{\log |\Delta_K|}{|\Delta_K|^{1/32- 1/\log \log |\Delta_K|} }+\frac{5800 R_F \log |\Delta_K|}{|\Delta_K|^{1/32}} + 20 R_F\right),
\end{split}
\end{equation*}

which translates into 

\begin{equation}\label{finalbh}
h_i^\infty \leq 0.7{4}+ 8.4R_F +\left(\frac{1}{10}+\frac{7200h_F^3 R_F}{|\Delta_K|^{1/32- 1/\log \log |\Delta_K|} }+\frac{2430 R_F }{|\Delta_K|^{1/32}}\right)\log|\Delta_K|.
\end{equation}

\section{Proof of the Main Theorem}\label{sec:main}
We prove now our main theorem, i.e. Theorem \ref{thm:main}.

\begin{proof}
First, Proposition \ref{easyhbound} and Equation \ref{def:h} give us:

\begin{equation}
\begin{split}
-(\frac{\gamma_F}{2}+\frac{1}{4}\log \Delta_F+\log 2\pi+\gamma_\mathbb{Q})+\frac{\sqrt{5}}{20}\log \Delta_K \leq\\ \frac{1}{60[k:\mathbb{Q}]}\log \operatorname{N}(\Delta^0_{min}(C)) +{\frac{1}{4}}\sum_{i=1}^{4}h_i^\infty-\frac{4}{5}\log 2-\log\pi 
\end{split}
\end{equation}

This together with Equation \ref{finalbh} under the assumption $\Delta_K\geq98^4$ gives

\begin{equation}
\begin{split}
 (\frac{\sqrt{5}}{20}-\frac{1}{10}-\frac{7200h_F^3 R_F}{|\Delta_K|^{1/32- 1/\log \log |\Delta_K|} }&-\frac{2430 R_F }{|\Delta_K|^{1/32}})\log \Delta_K\leq \\ \frac{1}{60[k:\mathbb{Q}]}\log \operatorname{N}(\Delta^0_{min}(C)) +  \frac{\gamma_F}{2}+\frac{1}{4}\log \Delta_F+\log 2\pi&+\gamma_\mathbb{Q}-\frac{4}{5}\log 2-\log\pi + 0.7{4}+ 8.4R_F \leq\\
 \frac{1}{60[k:\mathbb{Q}]}\log \operatorname{N}(\Delta^0_{min}(C)) +  \frac{\gamma_F}{2}+&\frac{1}{4}\log \Delta_F + 8.4R_F + 1.4{5}.
\end{split}
\end{equation}

For simplicity, write the previous expression as:
$$
(\mathcal{A}-\frac{\mathcal{B}}{|\Delta_K|^{1/32- 1/\log \log |\Delta_K|}}-\frac{\mathcal{C}}{|\Delta_K|^{1/32}})\log \Delta_K\leq h^0+\mathcal{D}.
$$
Assume $|\Delta_K|>e^{e^{64}}, (20\mathcal{C})^{32}, (20\mathcal{B})^{64}$. Then the term $\frac{\mathcal{B}}{|\Delta_K|^{1/32- 1/\log \log |\Delta_K|}}\leq\frac{\mathcal{B}}{|\Delta_K|^{1/64}}\leq \frac{1}{20}$ and $\frac{\mathcal{C}}{|\Delta_K|^{1/32}}\leq\frac{1}{20}$. So $\frac{1}{10}\cdot\log\Delta_K\leq h^0+\mathcal{D}$. Then $$\Delta_K\leq\max\{e^{e^{64}}, (20\mathcal{C})^{32},(20\mathcal{B})^{64}, e^{10(h^0+\mathcal{D})}\}.$$
Moreover, in our situation $\mathcal{C}<\mathcal{B}$ and the result follows.
\end{proof}

\section{An explicit example: the real field $\mathbb{Q}(\sqrt{5})$}\label{sec:comp}
In this section we fix once and for all the totally real field $F=\mathbb{Q}(\sqrt{5})$. We want to explicitly determine the finite list (\cite[Thm. 1.1]{HabeggerPazuki17}) of genus 2 curves defined over $\bar{\mathbb{Q}}$ with CM by the maximal order of a cyclic Galois CM field containing $F$ and with good reduction everywhere. Our Corollary \ref{cor:main} gives us a bound for the discriminant of such a field $K$. For each of these fields there are a finite number of such curves and we can compute them via the complex multiplication method \cite{Streng14}. For each curve we can check the potentially good reduction everywhere condition with Liu's criterion \cite[Thm. 1 (I)]{Liu93}.

\subsection{The discriminant bound} 

In the case under consideration we have $\Delta_F=5$, $h_F=1$, $R_F=\log \frac{\sqrt{5}+1}{2}$ and 
$\gamma_F\leq 2$ (\cite[p.19]{Ihara06}). If we ask  $\operatorname{N}(\Delta^0_{min}(C))=1$, we obtain by Corollary \ref{cor:main}, $\Delta_K\leq e^{e^{64}}$.
This is an explicit bound. However, too big to be practical as it will be discussed later.

\subsection{Listing CM fields}\label{subsec:fields}
In order to list all the quartic CM fields Galois over $\mathbb{Q}$ and containing $\mathbb{Q}(\sqrt{5})$ we first looked into the LMFDB data base \cite{lmfdb}. Unfortunately, this list is only certified to be complete up to discriminant $4\cdot 10^6$. In this interval we find 45 such fields: see the \href{https://www.lmfdb.org/NumberField/?count=None&hst=List&degree=4&signature=\%5B0\%2C2\%5D\&cm_field=yes\&galois_group=C4\&is_galois=yes\&discriminant=1-4000000\&subfield=x\%5E2-5\&search_type=List}{list}. 

After discussion with John Voight about the completeness of the list, he showed us how to produce the full list of such fields up to a given discriminant via the Kronecker-Weber Theorem: all abelian extensions $K$ of $\mathbb{Q}$ which
contain $\mathbb{Q}(\sqrt{5})$ are given by $K = \mathbb{Q}(\zeta_N)^H$
where $H \leq (\mathbb{Z}/N\mathbb{Z})^*$
is a subgroup whose projection to $(\mathbb{Z}/5^{v_5(N)}\mathbb{Z})^*$  is contained in the subgroup of squares $(\mathbb{Z}/5^{v_5(N)}\mathbb{Z})^{*2}$. We can dually think about Dirichlet characters
  $\chi : (\mathbb{Z}/N\mathbb{Z})^* \rightarrow \mathbb{C}^*$
under the correspondence $H = \operatorname{Ker} \chi$. 
The conductor-discriminant formula tells us $\operatorname{disc} K$ in terms of $N$ and
$\chi$ (and its powers). In addition, he provided Magma \cite{magma} code to us to produce the list of fields.  In a few hours it is possible to confirm the completeness of the LMFDB output for discriminant under $1,6\cdot10^9$.

\subsection{Computing the curves}\label{subsec:comp} Once we have the possible CM fields we use Streng's algorithm \cite{Streng14} and code in Sage \cite{sagemath} version 9.5 to compute the class polynomials corresponding to these fields. The simultaneous computation of the polynomials for the 45 fields of discriminant smaller than $4\cdot 10^6$ took several days on a server of the University of G\"ottingen. The parameters of the server are: 32 GB RAM, 4 AMD Opteron Processors 6272, 64 cores. They are listed in the ancillary file attached to this paper on arXiv \ref{}. 

We need to factor the class polynomials in order to check if we find curves whose absolute invariants as in \cite{Streng14} are algebraic integers and hence having everywhere potentially good reduction \cite{Liu93}. This factorization is straightforward since we get polynomials of relatively small degree.  From the Igusa invariants \cite{Igusa60} of the curves we reconstruct equations of the curves using Mestre algorithm \cite{Mestre}. 

As suspected, we find a small number of such curves and they have a very small discriminant.

\begin{prop}\label{curves} The only genus 2 curves over $\bar{\mathbb{Q}}$ with CM by a cyclic Galois CM field containing $\mathbb{Q}(\sqrt{5})$ with potentially good reduction everywhere and discriminant smaller than $4\cdot10^6$ are:
\begin{enumerate}
\item $y^2=x^5+1$ with CM by the splitting field of the polynomial $x^4+x^3+x^2+x+1=0$ with discriminant 125. 
\item The curve with absolute Igusa invariants \small $$(j_1,j_2,j_3)=(183708000, 474590099025000000, 25021491747613593750000000)$$ \normalsize and with CM by the splitting field of the polynomial $x^4 + 10x^2 + 20=0$ with discriminant 8.000. 
    \end{enumerate}
\end{prop}

An equation for the curve in Proposition \ref{curves}(2) is
{\fontsize{6pt}{7pt} \selectfont \begin{equation*}
\begin{split}
y^2= &4222312357426855589892446843449627238487267610385292053012303364265424735291376597624117160311275176522847679373169206x^6 + \\ &1006226725888251106846275261305954015390155673145305039217582666370 7943648243534530978294450045033943455794367399292714x^5 -\\ &
403057053963299190180254632598191705055375892684621358577773309197361984942    46121963162248732048009377132383269022406680x^4 -\\ &
609219508903329966709273429330080475708351005002803540326939976421883670959    54308915201749284254022162360041511895991180x^3 +\\ &
120103048094637088181226241040284584144279528614312313099089754500382331741
    514157897756951562714088003833840991511346620x^2 +\\ &
314602025756187181666775548578025941334222682912368207923016437797670514349
    02351086040699504238675524130596116583519890x +\\ &
194921686940333930721040072259071993554647979114132103254654330879089372270    7539347423541210530941057918022842861414235.
    \end{split}
    \end{equation*} }
\normalsize

Even if with big coefficients, this equation is in reduced form as can be checked with Stoll-Cremona algorithm \cite{Stoll}. These curves already appeared in \cite[Thm. 6.5]{KS23}.

We strongly suspect that there is no other genus $2$ curve  over $\bar{\mathbb{Q}}$ with CM by a cyclic Galois CM field containing $\mathbb{Q}(\sqrt{5})$ with potentially good reduction everywhere: notice that they have discriminant $125, 8000\ll 4\cdot10^6\ll e^{e^{64}}$.

\bibliography{Bib_igusa}
\bibliographystyle{alpha}
\end{document}